\documentclass[ECP]{ejpecp} 

\SHORTTITLE{Maximum of the reflected Brownian motion before its last zero} 

\TITLE{Probability that the maximum of the reflected Brownian motion over a finite interval $[0,t]$ is achieved by its last zero before $t$} 

\AUTHORS{%
  Lagnous Agn\`es\footnote{Institut de Math\'ematiques de Toulouse, UMR 5219, Universit\'e Toulouse 2, France.
    \BEMAIL{lagnoux@univ-tlse2.fr}}
  \and 
  Mercier Sabine\footnote{Institut de Math\'ematiques de Toulouse, UMR 5219, Universit\'e Toulouse 2, France.
    \BEMAIL{mercier@univ-tlse2.fr}}
		\and
		Vallois Pierre\footnote{Institut de Math\'ematiques Elie Cartan, INRIA-BIGS, CNRS UMR 7502, Universit\'e de Lorraine, France.
		\BEMAIL{pierre.vallois@univ-lorraine.fr}}
		}

\KEYWORDS{Reflected Brownian motion ; Bessel process ; Brownian meander and Brownian bridge ; Brownian excursions ;  Gamma function ; local score.} 

\AMSSUBJ{60 G 17 ;  60 J 25 ; 60 J 65.} 

\SUBMITTED{May 2015} 
\ACCEPTED{} 

\VOLUME{0}
\YEAR{2012}
\PAPERNUM{0}
\DOI{vVOL-PID}

\ABSTRACT{We calculate the probability $p_c$ that the maximum of a reflected Brownian motion $U$  is achieved on a complete excursion, i.e. $p_c:=P\big(\overline{U}(t)=U^*(t)\big)$ where $\overline{U}(t)$ (respectively $U^*(t)$) is  the maximum  of the process $U$ over the time interval $[0,t]$ (resp. $\big[0,g(t)\big]$ where $g(t)$ is the last zero of $U$ before $t$).}

\newcommand{\abs}[1]{\left\lvert #1 \right\rvert}

\newcommand{\beq} {\begin{eqnarray*}}
\newcommand{\eeq} {\end{eqnarray*}}

\newcommand{\noi} {\noindent}

\def \R{\mathbb{R}}
\def \C{\mathbb{C}}

\def \N{\mathbb{N}}
\def \Z{\mathbb{Z}}
\def \E{\mathbb{E}}
\def \P{\mathbb{P}}

\newcommand{\dis}{\displaystyle }

\def \leqp{\leqslant}
\def \geqp{\geqslant}

\newcommand{\1}{{1\hspace{-0.2ex}\rule{0.12ex}{1.61ex}\hspace{0.5ex}}}

\usepackage{lscape}

\newcommand{\promille}{%
  \relax\ifmmode\promillezeichen
        \else\leavevmode\(\mathsurround=0pt\promillezeichen\)\fi}
\newcommand{\promillezeichen}{%
  \kern-.05em%
  \raise.5ex\hbox{\the\scriptfont0 0}%
  \kern-.15em/\kern-.15em%
  \lower.25ex\hbox{\the\scriptfont0 00}}

\begin{document}



\section{Introduction}\label{sec:pc}

{\bf 1.1 Motivation.}  The local score of a biological sequence  is its "best" segment   with respect to some scoring scheme (see e.g. \cite{Wat95} for more details) and the knowledge of its distribution is important 
(see e.g. \cite{KAl90}, \cite{MD01}).  Let us briefly recall the mathematical setting while biological interpretations can be found in \cite{CLMV14}. Let $S_n:=\epsilon_1+\cdots+\epsilon_n $ be the  random walk generated by the sequence of the independent and identically distributed random variables  $(\epsilon_i,\, i\in \mathbb{N})$ that are centered with unit variance.
%
%
The local score is the process: $\dis U_n:=S_n-\min\limits_{0\leqp i\leqp n}S_i$, where $n\geqp 0$.
%
\noi The path of  $(U_n,\, n\in \N)$ is a succession of $0$ and excursions above $0$. In \cite{CLMV14}, the authors only took into consideration   {\it complete} excursions up to a fixed time $n$ and so considered
the maximum $U^*_n$ of the heights of all the complete excursions up to time $n$ instead of the maximum $\overline U_n$ of the path until time $n$. They also introduced the random time $\theta^*$ of the length of the segment that realizes $U^*_n$.
Since it is easy to simulate $(S_k,\, 0\leqp k\leqp n)$, for any $n$ not too large, we get an approximation of the law of $(U^*_n, \theta^*_n)$ for a given $n$. Simulations have  shown that for an important proportion of sequences, $\overline U_n$
is realized during the last incomplete excursion. As expected, the number of excursions naturally increases when the length of the sequence growths, however the proportion of sequences that achieve their maximum on a complete excursion remains strikingly  constant.
The main goal of this study is to explain these observations  and to calculate this probability when $n$ is large, see Proposition \ref{pNY} below.

{\bf 1.2 Link with the Brownian motion.} According to the functional convergence theorem  of Donsker, the random walk $(S_k,\; 0\leqp k\leqp n)$ (resp.
$(U_k,\; 0\leqp k\leqp n)$) normalized by the factor $1/\sqrt{n}$ converges in distribution, as $n\rightarrow\infty$, to the Brownian motion (resp. the reflected Brownian motion). We prove (see Theorem \ref{tTu1} for a precise formulation) that the probability that the maximum of a reflected Brownian motion over a finite interval $[0,t]$ is achieved on a complete excursion is around 30$\%$ and is thus independent of $t$. This result permits to answer to the two questions asked in the discrete setting, when $n$ is large.



\noi
Let $U$ be the reflected Brownian motion started at $0$, i.e. $U_t=|B_t|$ where $(B_t)_t$ is the standard one-dimensional Brownian motion started at $0$. In Chabriac \textit{et al.} \cite{CLMV14}, the authors have considered two maxima: $\dis \overline{U}(t)$ and $U^*(t)$, the first (resp. second)  one being the maximum of $U$ up to time $t$ (resp. the  last zero before $t$), namely  $\dis \overline{U}(t):=\max_{0\leqp s\leqp t}U_s$ and $U^*(t)=\overline{U}\big(g(t)\big)$, where $g(t):= \sup\{s\leqp t, U(s)=0\}$. In \cite{CLMV14}, the density function of the pair $\big(U^*(t),\theta^*(t)\big)$ has been calculated where $\theta^*(t)$ is the first hitting time of level $U^*(t)$ by the process $\big(U_s, 0\leqp s\leqp g(t)\big)$. Here we only deal with $U^*(t)$ and $\overline{U}(t)$.


It is clear that $\overline{U}(t)=U^*(t)$ if and only if $\dis  U^{**}(t) \leqp \overline{U}(t)$, where
\begin{equation}\label{Tur1}
    U^{**}(t):=\max_{g(t)\leqp s \leqp t} U(s).
\end{equation}
In that case, the maximum of $U$ over $[0, t]$ is the maximum of all the complete excursions of $U$ which hold before $t$. We introduce the probability $p_c$ that the maximum of $U$ over $[0, t]$ is achieved on a complete excursion:
\begin{equation}\label{def:pc}
p_c=\P\left(\overline{U}(t)=U^*(t)\right)=\P\left(U^*(t)>U^{**}(t)\right),
\end{equation}
Let $\psi$ be the logarithmic derivative of the Gamma function:
\begin{equation}\label{Bor3}
   \psi(x):=\Gamma'(x)/\Gamma(x).
\end{equation}
%
The main result of our study is

\begin{theorem}\label{tTu1}
The probability $p_c$ equals $\psi\left( 1/4\right) - \psi\left( 1/2\right)+1+\pi/2\approx 0.3069$.
\end{theorem}
{\bf 1.3 Back to discrete sequences.} We now go back to the setting of random walks introduced in paragraph 1.1. Let $p_c^{(n)}$ be  the probability that the maximum of $\big(U_k, 0\leqp k\leqp n\big)$ is achieved on a complete excursion, namely
\begin{equation}\label{NY1}
    p_c^{(n)}:=P\Big(\max_{0\leqp k\leqp n}U_k=\max_{0\leqp k\leqp g_n}U_k\Big)\quad \mbox{ where }g_n:=\max\{k\leqp n, U_k=0\}.
\end{equation}
%
\begin{proposition}\label{pNY} $p_c^{(n)}$ converges to $p_c$ as $n\rightarrow \infty$.
\end{proposition}

\noi The convergence of $p_c^{(n)}$ can be obtained from Theorem 3.3 in \cite{CLMV14} and the fact that  
the event $N$ defined by (4.23) in \cite{CLMV14} is actually included in $\Big\{\max_{0\leqp k\leqp n}U_k=\max_{0\leqp k\leqp g_n}U_k\Big\}$.

{\bf 1.4 Main steps of the proof.} We now consider the Brownian motion setting. The density function of $\overline{U}(t)$ is known (see either Subsection 2.11 in  \cite{Billingsley99} or Lemma 3.2 in \cite{RVY08}) and the one of $U^*(t)$ has been calculated in \cite{CLMV14}. Obviously, the knowledge of the  distributions of $\overline{U}(t)$ and $U^*(t)$ is not sufficient to determine $p_c$. The trajectory of $(U_s, 0\leqp s\leqp t)$ naturally splits in two parts before and after the random time $g(t)$ which is not a stopping time. Although $\big(U_s, 0\leqp s\leqp g(t)\big)$ and $\big(U_s, g(t)\leqp s\leqp t\big)$ are not independent, the scaling with $ g(t)$ leads to independence. Indeed,

\begin{align}\label{eq:meandre_indep}
&\left(\frac{1}{\sqrt{g(t)}}U(sg(t)),\; 0\leqp s \leqp 1\right),\; \left(\frac{1}{\sqrt{t-g(t)}}\abs{B(g(t)+s(t-g(t))},\; 0\leqp s \leqp 1\right), \; g(t)
\end{align}
are independent.
Moreover each part of the above triplet has a known distribution. The process $\left(g(t)^{-1/2}B(g(t)s),\; 0\leqp s\leqp 1\right)$ is distributed as the Brownian bridge $\big(b(s),\; 0\leqp s\leqp 1\big)$, (see e.g. \cite{Bertoin96}) and the second component in \eqref{eq:meandre_indep} is the Brownian meander denoted $m$. The scaling property of the Brownian motion implies that $g(t)$ is distributed as $t g(1)$ while the distribution of $g(1)$ is the arcsine one (see again \cite{Bertoin96}):
\begin{equation}\label{eq:dist_g1}
    \P(g(1)\in dx)=\frac{1}{\pi\sqrt{x(1-x)}}\1_{[0,1]}(x)\,dx.
\end{equation}
Consequently,
\begin{align}\label{eq:U*/U**}
\left(U^*(t),U^{**}(t)\right)&\overset{(d)}{=}\left(\sqrt{tg(1)}b^*,\sqrt{t(1-g(1))} \max_{0\leqp u\leqp 1} m(u)\right)
\end{align}
where $\dis b^*:=\sup_{0\leqp s\leqp 1} |b(s)|$. Its distribution function  is given by the Kolmogorov-Smirnov formula   (see e.g. \cite{PY01}):
\begin{equation}\label{eq:dist_b*}
    \P(b^*>x)=2\sum_{k\geqp 1} (-1)^{k-1} e^{-2k^2x^2}, \quad x>0.
\end{equation}
Formula \eqref{eq:U*/U**} permits to determine the law of $\big(U^*(t),U^{**}(t)\big)$, once we know the distribution of
$\dis \max_{0\leqp u\leqp 1} m(u)$. But by \cite{BLGY87}, for any bounded Borel function $f$,
\begin{align}
\E[f(m(u),\; 0\leqp u\leqp 1)]&=\sqrt{\frac{\pi}{2}}\E\left[\frac{1}{R(1)}f(R(u),\; 0\leqp u \leqp 1)\right],
\end{align}
where $(R(u),0\leqp u\leqp 1)$ stands for a 3-dimensional Bessel started at 0.

\noi Due to the scaling property \eqref{eq:U*/U**}, we deduce that $p_c$ does not depend on $t$ and
\begin{equation}\label{Bor1}
p_c =\sqrt{\frac{\pi}{2}}\E\left[F\left(b^*\sqrt{\frac{g(1)}{1-g(1)}}\right)\right]
\end{equation}
where
\begin{align}\label{def:F}
F(x) &:=  \E\left[\frac{1}{R(1)}\1_{\left\{\max_{0\leqp u \leqp 1}R(u)<x\right\}}\right].
\end{align}

According to \eqref{Bor1} we have first to determine the function $F$ (see Lemma \ref{lem:F} below), second the distribution function of
$\dis b^*\sqrt{\frac{g(1)}{1-g(1)}}$ (see Proposition \ref{prop:dens}) and third to calculate the expectation. The details are given in Section \ref{sBor1}.

\section{Proof of Theorem \ref{tTu1}}\label{sBor1}

\setcounter{equation}{0}

\begin{lemma}\label{lem:F}
For any $x>0$,
\begin{align}\label{eq:F}
F(x)&= \sqrt{\frac{2}{\pi}} \sum_{k\in \Z} \left\{ e^{-2k^2x^2} - e^{-(2k+1)^2x^2/2}\right\}=\frac{4}{x} \sum_{k\geqp 0} \exp\left\{-\frac{(2k+1)^2\pi^2}{2x^2} \right\}.
\end{align}
\end{lemma}

\begin{proof}
First, by \cite[formula 1.1.8, p317]{BS02},
\begin{align}\label{eq:R_1}
\P\left(\max_{0\leqp u \leqp 1}R(u)<y,\; R(1)\in dz\Big\vert R(0)=x \right)=\frac{z}{x\sqrt{2\pi}} \times S \times \1_{\left\{y>x, z<y\right\}}dz
\end{align}
where
\begin{align*}
S=\sum_{k\in \Z} \left[\exp\left\{-\frac{(z-x+2ky)^2}{2}\right\}-\exp\left\{-\frac{(z+x+2ky)^2}{2}\right\}\right].
\end{align*}
A Taylor expansion of $x\mapsto \exp\left\{-\frac{(z\pm x+2ky)^2}{2}\right\}$ at $x=0$ leads to

\begin{align*}
\P\left(\max_{0\leqp u \leqp 1}R(u)<y,\; R(1)\in dz\Big\vert R(0)=0 \right)=\frac{2z}{\sqrt{2\pi}} \left[\sum_{k\in \Z} (z+2ky) \exp\left\{-\frac{(z+2ky)^2}{2}\right\} \right]\1_{\left\{z<y\right\}}dz.
\end{align*}
As a consequence,
\begin{align*}
&F(y)= \frac{2}{\sqrt{2\pi}} \sum_{k\in \Z} \int_0^y  (z+2ky) \exp\left\{-\frac{(z+2ky)^2}{2}\right\}  dz\\
&= \frac{2}{\sqrt{2\pi}} \sum_{k\in \Z} \left[ \exp\left\{-2k^2y^2\right\}  -\exp\left\{-\frac{(2k+1)^2y^2}{2}\right\}\right].
\end{align*}
The second equality in \eqref{eq:F} is a direct consequence of the Poisson summation formula (see, e.g. \cite[Chap. XIX p.630]{Feller71b}:
\begin{align*}
\frac{1}{\sqrt{2\pi t}}\sum_{k\in \Z} & \left[\exp\Big\{-\frac{1}{2t}\big(x_0+2k(\beta_0-\alpha_0)\big)^2\Big\}- \exp\Big\{-\frac{1}{2t}\big(2\beta_0-x_0+2k(\beta_0-\alpha_0)\big)^2\Big\}\right]\\
&=\frac{1}{\beta_0-\alpha_0}\sum_{k\geqp 1} \left[ \left(\cos\left(\frac{k\pi x_0}{\beta_0-\alpha_0}\right)-\cos\left(\frac{k\pi(2\beta_0- x_0)}{\beta_0-\alpha_0}\right)\right)\exp\left\{-\frac{k^2\pi^2t}{2(\beta_0-\alpha_0)^2}\right\}\right]
\end{align*}
applied with $t=1$, $x_0=0$, $\beta_0=x/2$ and  $\alpha_0=-x/2$.
\end{proof}

\begin{proposition}\label{prop:dens}
For any $x>0$,
\begin{align}\label{eq:dens}
\P\left(b^*\sqrt{\frac{g(1)}{1-g(1)}}>x \right)=\frac{2}{\pi} \int_0^{\infty} A(u) e^{-2x^2u}\frac{du}{\sqrt u},
\end{align}
where
\begin{align}\label{def:A}
A(u)&:=\sum_{k\geqp 1}  (-1)^{k-1}\frac{k}{k^2+u}.
\end{align}
\end{proposition}


\begin{proof}
We introduce  a cut-off  $0<\varepsilon<1$ and we define:

\begin{equation}\label{Bor2}
 \varphi_{\varepsilon}(x):=\P\left(g(1)<1-\varepsilon,\; b^*\sqrt{\frac{g(1)}{1-g(1)}}>x \right).
\end{equation}
Using the independence between $g(1)$ and $b^*$,  \eqref{eq:dist_g1} and \eqref{eq:dist_b*}, we deduce:
$$\varphi_{\varepsilon}(x)=\frac{1}{\pi} \int_0^{1-\varepsilon} \frac{1}{\sqrt{y(1-y)}}\P\left(b^*>x \sqrt{\frac{1-y}{y}}\right)dy=\frac{2}{\pi} \sum_{k\geqp 1} I_k(\varepsilon)$$


where $\dis I_k(\varepsilon):=(-1)^{k-1} \int_0^{1-\varepsilon} \frac{1}{\sqrt{y(1-y)}}\exp\left\{-\frac{2k^2x^2(1-y)}{y}\right\}dy$. The inversion of the sum and the integral is available since $(1-y)/y \geqp \varepsilon'>0$ where  $\varepsilon':=\varepsilon/(1-\varepsilon)$. Making the change of variables
 $(1-y)/y=u/k^2$  leads to:
$$I_k(\varepsilon)=\frac{(-1)^{k-1}}{k} \int_{\varepsilon'k^2}^{\infty} \frac{1}{\sqrt{u}(1+u/k^2)}\exp\left\{-2x^2u\right\}du.$$


The identity $\dis \frac{1}{1+u/k^2}=1-\frac{u}{k^2(1+u/k^2)}$ allows to invert the sum and the integral. Finally we get:


\begin{align*}
\varphi_{\varepsilon}(x)&=\frac{2}{\pi} \int_{\varepsilon'}^{\infty} \left(\sum_{\varepsilon' k^2 \leqp u} \frac{(-1)^{k-1}}{k}  \frac{1}{1+u/k^2}\right)\exp\left\{-2x^2u\right\}\frac{du}{\sqrt{u}}
=\frac{2}{\pi} \int_{\varepsilon'}^{\infty} S_{n(\varepsilon', u)}(u) \frac{\exp\left\{-2x^2u\right\}}{\sqrt{u}}du
\end{align*}
with $n(\varepsilon',u)=\lfloor \sqrt{{u}/{\varepsilon'}} \rfloor$, $\dis S_{n}(u):=\sum_{k=1}^n (-1)^{k-1} \phi(1/k,u)$ and $\phi(y,u):=y/(1+uy^2)$.

Note that $\dis \left|\frac{\partial \phi}{\partial y} \right|\leqp 1$, then, considering $n=2m$ and $n=2m+1$ and using the mean value inequality we obtain:
\begin{equation}\label{Bor15}
    \abs{S_{2m}(u)}=\abs{\sum_{k'=1}^m  \phi\left(\frac{1}{2k'-1},u\right)-\phi\left(\frac{1}{2k'},u\right)}
\leqp \sum_{k'=1}^m \abs{\frac{1}{2k'-1}-\frac{1}{2k'}}\leqp \sum_{k'=1}^{\infty} \frac{1}{2k'(2k'-1)}<\infty.
\end{equation}
Similarly, $\abs{S_{2m+1}(u)}<\infty$. Since $\varepsilon'\to 0$ and $n(\varepsilon',u)\to 0$ as $\varepsilon$ goes to 0, then, identity \eqref{eq:dens} is a direct consequence of the Lebesgue dominated convergence theorem.
\end{proof}

Lemma \ref{lem:F} and Proposition \ref{prop:dens} allow to obtain a new integral form for $p_c$.

\begin{lemma}\label{lBor1}
One has
\begin{align*}
p_c&= 8\int_{0}^{\infty} \frac{uA(u^2)}{\sinh(2\pi u)}du,
\end{align*}
\end{lemma}

\begin{proof}
We deduce easily from \eqref{def:A} that
$$A(u)=\sum_{k'\geqp 1} \phi\left(\frac{1}{2k'-1},u\right)-\phi\left(\frac{1}{2k'},u\right),$$
where $\phi(y,u):=y/(1+uy^2)$. Then inequality \eqref{Bor15} implies that $\sup _{u\geqp 0}A(u)<\infty$. By Lemma \ref{lem:F}, Proposition \ref{prop:dens}, the definition \eqref{def:A} of $A$ and the Fubini theorem, we get

\begin{align*}
p_c&=\sqrt{\frac{\pi}{2}}\frac{32}{\pi} \sum_{k\geqp 0} \int_0^{\infty} \sqrt{u} A(u) \left(\int_0^{\infty} \exp\left\{-\frac{(2k+1)^2\pi^2}{2x^2}-2x^2u\right\}dx\right) du.
\end{align*}
But making $s=2x^2u$ and letting $z=2(2k+1)\pi\sqrt u$, we get:

\begin{align*}
\int_0^{\infty} \exp\left\{-\frac{(2k+1)^2\pi^2}{2x^2}-2x^2u\right\}dx&=\frac{1}{\sqrt{2u}} \left(\frac{z}{2}\right)^{1/2}K_{1/2}(z)
\end{align*}
where (cf \cite[Formula (15) p183]{Watson95})
$$K_{\nu}(x)=\frac{1}{2}\Big(\frac{x}{2}\Big)^\nu\int_0^\infty\exp\Big\{-s-\frac{x^2}{4s}\Big\}ds.$$
Recall that  $K_{\nu}=K_{-\nu}$ \cite[Formula (8) p79]{Watson95} and $\dis K_{-1/2}(x)=\sqrt{\frac{\pi}{2x}}e^{-x}$ \cite[Formula (13) p80]{Watson95} . It follows that
%



\begin{align*}
p_c&=8 \sum_{k\geqp 0} \int_0^{\infty} A(u) e^{-2(2k+1)\pi\sqrt u} du=4  \int_0^{\infty} \frac{A(u)}{\sinh{(2\pi\sqrt u)}} du=8  \int_0^{\infty} \frac{vA(v^2)}{\sinh{(2\pi v)}} dv.
\end{align*}

\end{proof}

We now focus on the function $A$. Our method is based on the  crucial fact that $A$ can be expressed with the function $\psi$  defined by \eqref{Bor3}.

\begin{lemma}\label{lem:A}
\begin{enumerate}
  \item We have:

\begin{align}\label{eq:A}
A(u)&=\frac 14 \left[\psi\left(\frac{i\sqrt u}{2}\right)+\psi\left(-\frac{i\sqrt u}{2}\right)-\psi\left(\frac{1+i\sqrt u}{2}\right)-\psi\left(\frac{1-i\sqrt u}{2}\right)\right], \quad u\geqp 0.
\end{align}
  \item There exists $a$, $b>0$ such that
\begin{align}\label{ineg:psi}
\abs{\psi(z)}\leqp a + b\abs{z}^2,\quad \forall\; z\in\mathbb{C},\; \abs{Im\; z} \geqp 1.
\end{align}
\end{enumerate}
Consequently, $ p_c= I_2-I_1$, where
\begin{equation}\label{Bor5}
I_k:=2\int_{0}^{\infty} \frac{vF_k(v)}{\sinh(2\pi v)}dv,\; k=1,2
\end{equation}
and $F_1(v):=\psi\left(\frac{1+iv}{2}\right)+\psi\left(\frac{1-iv}{2}\right)$,
$F_2(v):=\psi\left(\frac{iv}{2}\right)+\psi\left(-\frac{iv}{2}\right)$.

\end{lemma}

\begin{proof} Formula \eqref{eq:A} and inequality \eqref{ineg:psi}) are a direct consequence of  identity (3), p 15 in \cite{Bateman81}, i.e.

\begin{align}\label{eq:expansion_psi}
\psi(z)&=-\gamma-\frac 1z +\sum_{n\geqp 1} \frac{z}{n(z+n)}=-\gamma -\frac 1z +\Big(\sum_{n\geqp 1} \frac{1}{n^2}\Big)z  -\sum_{n\geqp 1} \frac{z^2}{n^2(z+n)}
\end{align}
with $\dis \gamma:=\lim_{m\to \infty} \Big(\sum_{k= 1}^m \frac{1}{k}-\ln m\Big)$.
\end{proof}

Due to the form of the functions $F_1$ and $F_2$, the integrals $I_1$ and $I_2$ can be viewed as integrals over a straight line in the plane. More precisely, we have:

\begin{lemma}\label{lBor2} $I_1$ and $I_2$ can be written as:

\begin{equation}\label{Bor6}
  I_1=-8i  \int_{\Delta_{1/2}} \frac{z-1/2}{\sin{(4\pi z)}} \psi(z)dz,\quad
  I_2=-8i  \int_{\Delta_{0}} \frac{z\psi(z+1)}{\sin{(4\pi z)}} dz
    \end{equation}

where, for any $a\in \R$, $\Delta_a$ is the line with parametrization  by $z=a+it$, $t\in\R$.
\end{lemma}

\begin{proof}
1) We note that $\sin(2\pi iv)=i\sinh(2\pi v)$. Thus
\begin{align*}
2\int_0^{\infty} \frac{v}{\sinh(2\pi v)} \psi\left(\frac{1+iv}{2}\right)dv&=
\frac{8}{i}\int_0^{\infty} \frac{\frac{1+iv}{2}-\frac 12}{\sin\left(4\pi i \left(\frac{1+iv}{2}\right)\right)} \psi\left(\frac{1+iv}{2}\right)d\left(\frac{1+iv}{2}\right)\\
&=-8i\int_{\Delta'_{1/2}} \frac{z-1/2}{\sin\left(4\pi z\right)} \psi\left(z\right)dz
\end{align*}
where $\Delta'_a$ is the half-line: $z=a+it$, $t\geqp 0$. Similarly,
\begin{align*}
2\int_0^{\infty} \frac{v}{\sinh(2\pi v)} \psi\left(\frac{1-iv}{2}\right)dv&=-8i\int_{\Delta''_{1/2}} \frac{z-1/2}{\sin\left(4\pi z\right)} \psi\left(z\right)dz
\end{align*}
where $\Delta''_a:=\{z=a+it, t\leqp 0\}$ and $a\in \R$. This implies the value of $I_1$ given by \eqref{Bor6}.

2) Formula \eqref{eq:expansion_psi} tells us that $\hat{\psi}(z):=\psi(z)+\frac 1z$ has no singularity at $z=0$.  Thus we study
\begin{align*}
2 \int_0^{\infty} \frac{v}{\sinh{(2\pi v)}} \hat{\psi}\left(\frac{iv}{2}\right)dv
&= \frac{8}{i} \int_0^{\infty} \frac{iv/2}{\sin{\left(4\pi \frac{iv}{2}\right)}} \hat{\psi}\left(\frac{iv}{2}\right)d\left( \frac{iv}{2}\right)
= -8i \int_{\Delta'_{0}} \frac{z}{\sin\left(4\pi z\right)} \hat{\psi}\left(z\right)dz
\end{align*}
and similarly
\begin{align*}
2 \int_0^{\infty} \frac{v}{\sinh{(2\pi v)}} \hat{\psi}\left(-\frac{iv}{2}\right)dv= 2 \int_0^{\infty} \frac{iv}{\sin{(2\pi i v)}} \hat{\psi}\left(\frac{iv}{2}\right)dv= -8i \int_{\Delta''_{0}} \frac{z}{\sin\left(4\pi z\right)} \hat{\psi}\left(z\right)dz.
\end{align*}
The identity (formula (8) p 16 in \cite{Bateman81}) :
\begin{equation}\label{Bor7}
  \hat{\psi}(z)=\psi(z)+\frac 1z=\psi(z+1)
\end{equation}
implies $\dis
\psi\left(\frac{iv}{2}\right)+\psi\left(-\frac{iv}{2}\right)=\hat{\psi}\left(\frac{iv}{2}\right)+\hat{\psi}\left(-\frac{iv}{2}\right)
$
and finally \eqref{Bor6}.
\end{proof}

We show in the following lemma that $I_1$ and $I_2$ are integrals over the vertical line.

\begin{lemma}\label{lem:I_1-I_2}
Let $0<\varepsilon <1/4$, then
\begin{equation}\label{Bor8}
    I_1=-8i \int_{\Delta_{1/2+\varepsilon}} \frac{z-1/2}{\sin{(4\pi z)}} \psi(z)dz,\quad
    I_2=-8i \int_{\Delta_{1/2+\varepsilon}} \frac{z\psi(z+1)}{\sin{(4\pi z)}}  dz+\psi\left( 1/4\right) - 2 \psi\left( 1/2\right).
\end{equation}
\end{lemma}

\begin{proof} We only deal with $I_2$, the proof related to $I_1$ is similar and easier.

 The quantity $\sin(4\pi z)$ cancels at $z=k/4$ for every $k\in \Z$, the zeros are simple. From \eqref{eq:expansion_psi}, we deduce that $h(z):=z\psi(z+1)/\sin{(4\pi z)}$ is meromorphic in
 $\left\{z\in \C;\; -1/4<Re z <3/4\right\}$ with poles at $1/4$ and $1/2$. We introduce the contour defined in Figure \ref{fig:contour}.

\begin{figure}[h]
\centering
\includegraphics[scale=0.3]{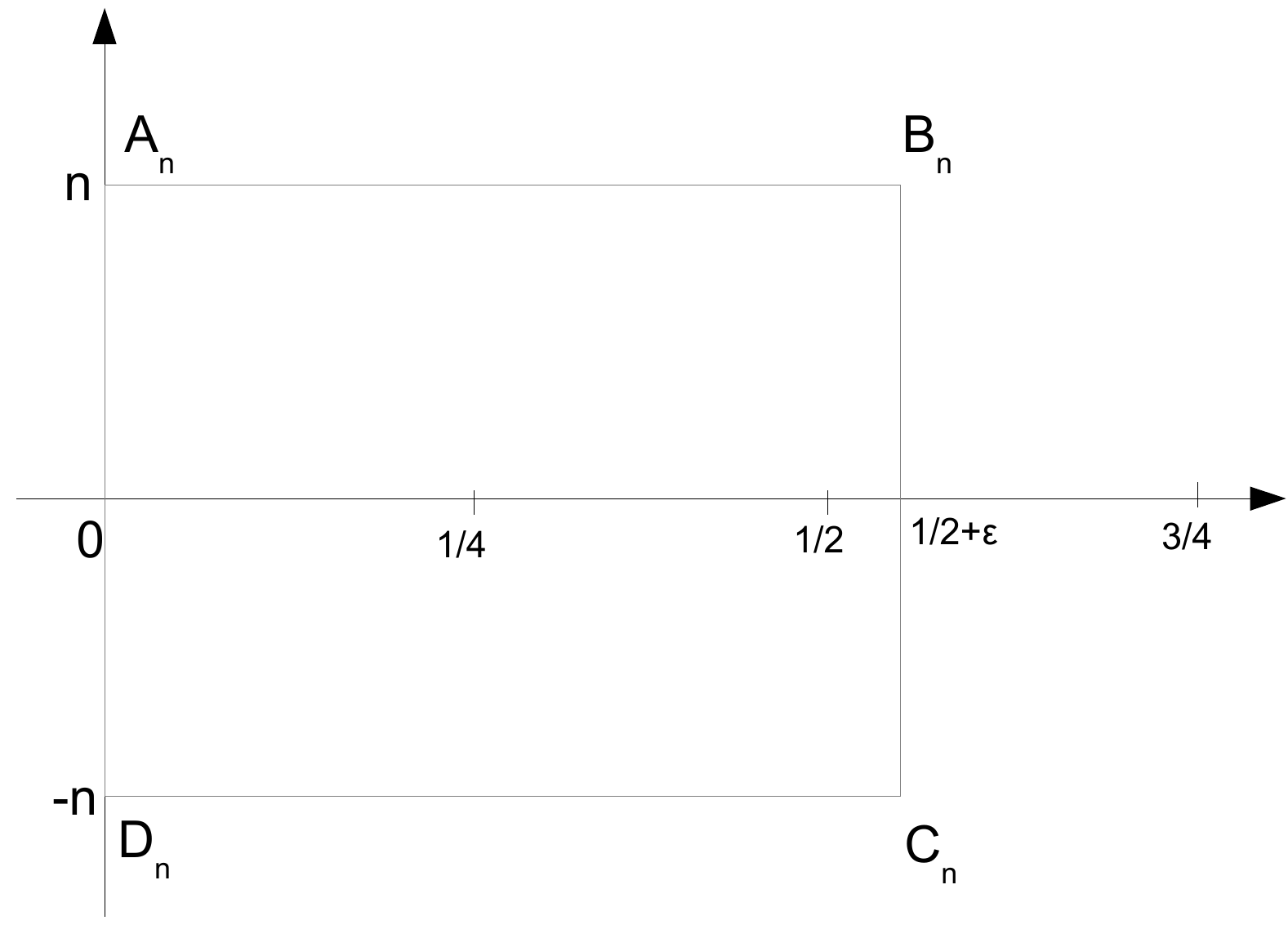}
\label{fig:contour}
\end{figure}

Then the residue theorem gives
\begin{align}\label{eq:residu_h}
\int_{C_nB_n}h(z)dz-\int_{D_nA_n}h(z)dz+\int_{D_nC_n}h(z)dz-\int_{B_nA_n} h(z)dz&=2i\pi \left\{Res\left(h,1/4\right)+Res\left(h,1/2\right)\right\}.
\end{align}
The residual at $1/4$ is given by
\begin{align*}
Res\left(h,1/4\right)&=\lim_{z\to 1/4}\Big\{ h(z)(z-1/4)\Big\}=\frac 14 \psi\left(5/4\right)\Big( \lim_{z\to 1/4} \frac{z-1/4}{\sin{(4\pi z)}-\sin{(4\pi 1/4)}}\Big)
&=-\frac{1}{16\pi}\psi\left(5/4\right)
\end{align*}
Using \eqref{Bor7} with $z=1/4$, we get $\dis Res\left(h,1/4\right)=-\frac{1}{4\pi}\left(1+\frac14 \psi\left(1/4\right)\right)$. Similarly, $\dis Res\left(h,1/2\right)=\frac{1}{4\pi}\left(1+ \psi\left(1/2\right)/2\right)$.

Now, we let $n$ goes to infinity. Inequality \eqref{ineg:psi} implies
\begin{align*}
\lim_{n\to \infty} \int_{A_nB_n} h(z)dz=\lim_{n\to \infty} \int_{C_nD_n} h(z)dz=0.
\end{align*}
Indeed, it follows from  the  parametrization of $A_nB_n$ of the type  $z=in+t$, inequality $\abs{in+t}\leqp n+1$ valid for
 $0\leqp t \leqp \varepsilon +1/2$, and
\begin{align*}
\big|\sin(4 \pi (in+t))\big|^2
&=\sinh(4 \pi n)^2 \cos(4 \pi t)^2+ \cosh(4 \pi n)^2 \sin(4 \pi t)^2\\
&\geqp \min\left\{\sinh(4 \pi n)^2,\cosh(4 \pi n)^2\right\}.
\end{align*}


We proceed analogously on $C_nD_n$. As a consequence, letting $n\to\infty$ in \eqref{eq:residu_h}, we get
\begin{align*}
I_2
&= -8i \left\{\int_{\Delta_{1/2+\varepsilon}} h(z)dz-2i\pi \left[-\frac{1}{4\pi}\left(1+ \psi\left(1/4\right)/4\right)+\frac{1}{4\pi}\left(1+ \psi\left(1/2\right)/2\right)\right]\right\}\\
&= -8i \int_{\Delta_{1/2+\varepsilon}} h(z)dz+ \psi\left(1/4\right)-2 \psi\left(1/2\right).
\end{align*}

\end{proof}

Bringing together  Lemmas \ref{lem:A}, \ref{lem:I_1-I_2} leads to
\begin{align*}
p_c&=-8i\int_{\Delta_{1/2+\varepsilon}} \frac{1}{\sin{(4\pi z)}} \left(z \psi(z+1)-(z-\frac 12) \psi(z)\right) dz+ \psi\left( 1/4\right) - 2 \psi\left( 1/2\right).
\end{align*}
Setting $z=1/2+u$ and using identity \eqref{Bor7} with $u+1/2$ instead of $z$ gives:
\begin{eqnarray}
p_c&=\dis -8i\int_{\Delta_{\varepsilon}}  h^+(u) du+ \psi\left( 1/4\right) - 2 \psi\left( 1/2\right),\label{Bor12}
\end{eqnarray}
where  $\dis h^{\pm}(z):=\frac{1}{\sin{(4\pi z)}} \left(1+\frac 12 \psi(\frac 12\pm z)\right)$.

We are not able to calculate directly $\dis \int_{\Delta_{\varepsilon}}  h^+(u) du$, however it is possible for
 $\dis \int_{\Delta_{\varepsilon}}  h^+(u) du\pm  \int_{\Delta_{\varepsilon}}  h^-(u) du$.

 \begin{lemma}\label{lBor3} For any $\varepsilon \in]0,1/4[$,
 \begin{eqnarray}
   \int_{\Delta_{\varepsilon}}  h^+(u) du + \int_{\Delta_{\varepsilon}}  h^-(u) du &=& \frac{i}{2}\left(1+\psi(1/2)/2\right)
   \label{Bor9}\\
   \int_{\Delta_{\varepsilon}}  h^+(u) du - \int_{\Delta_{\varepsilon}}  h^-(u) du &=&
   \frac{\pi}{2}\int_{\Delta_{\varepsilon}}\frac{\tan(\pi z)}{\sin(4\pi z)}dz. \label{Bor10}
 \end{eqnarray}
 \end{lemma}

\begin{proof}
1) We begin proving \eqref{Bor9}. The function $h^+$ is  meromorphic in $\left\{z\in \C;\; -1/4<Re z <1/4\right\}$ with unique pole $z=0$.  Then the residue theorem yields
$$
\int_{\Delta_{\varepsilon}}h^+(z)dz=\int_{\Delta_{-\varepsilon}}h^+(z)dz+2i\pi Res\left( h^+,0\right)=-\int_{\Delta_{\varepsilon}}h^-(z)dz+\frac{i}{2}\left(1+\psi(1/2)/2\right).
$$
2) Formula \eqref{Bor10} is a direct consequence of formula 11 p16  in \cite{Bateman81}: $\psi\left(\frac 12+z\right)-\psi\left(\frac 12-z\right)=\pi\tan(\pi z)$.
\end{proof}

It is easy to deduce from \eqref{Bor9} and \eqref{Bor10} that:
\begin{align*}
2\int_{\Delta_{\varepsilon}}h^+(z)d&z=\frac{\pi}{2}\int_{\Delta_{\varepsilon}}\frac{\tan(\pi z)}{\sin(4\pi z)}dz+\frac{i}{2}\left(1+\psi(1/2)/2\right).
\end{align*}

Relation \eqref{Bor12} implies directly that $p_c$ equals $\psi\left( 1/4\right) - \psi\left( 1/2\right)+2-2\pi i\alpha(\varepsilon)$, where $\dis \alpha(\varepsilon):=\int_{\Delta_{\varepsilon}}\frac{\tan(\pi z)}{\sin(4\pi z)}dz$.
Since the real number $p_c$ does not depend on $\varepsilon$, letting $\varepsilon\to 0$ leads to:
\begin{equation}\label{Bor14}
    p_c=\psi\left( 1/4\right) - \psi\left( 1/2\right)+2-2\pi i\alpha(0+),
\end{equation}

where  $\dis \alpha(0+)=\int_{\Delta_{0}}\frac{\tan(\pi z)}{\sin(4\pi z)}dz$. This integral is easy to calculate.

\begin{align*}
\alpha(0+)&=i \int_{\R}\frac{\tanh(\pi x)}{\sinh(4\pi x)}dx=\frac{i}{4}
\int_{\R}\frac{dx}{\cosh^2(\pi x)\cosh(2\pi x)}.
\end{align*}
We make the change of variable $u=\tanh(\pi x)$:
\begin{align*}
\alpha(0+)&=\frac{i}{2\pi} \int_{0}^1\frac{1-u^2}{1+u^2}du=\frac{i}{2\pi}\left(-1+2 \int_{0}^1\frac{du}{1+u^2}\right)=\frac{i}{2\pi}\left(-1+\frac{\pi}{2}\right).
\end{align*}

Theorem \ref{tTu1} follows from \eqref{Bor14} and the above result.







\end{document}